\documentclass{amsart}%
\usepackage{amsfonts}
\usepackage{amsmath}
\usepackage{amssymb}
\usepackage{graphicx}%
\providecommand{\U}[1]{\protect\rule{.1in}{.1in}}
\newtheorem{theorem}{Theorem}
\theoremstyle{plain}

\newtheorem{corollary}{Corollary}

\newtheorem{example}{Example}

\newtheorem{lemma}{Lemma}

\newtheorem{remark}{Remark}

\numberwithin{equation}{section}
\begin{document}
\title[ ]{Quantitative Korovkin theorems for sublinear, monotone and strongly translatable operators in $L_{p}([0, 1]), 1\le p\le +\infty$}
\author{Sorin G. Gal}
\address{Department of Mathematics and Computer Science\\
University of Oradea\\
University\ Street No. 1, Oradea, 410087, Romania\\
and Academy of Romanian Scientists, Splaiul Independentei nr. 54, 050094,
Bucharest, Romania}
\email{galso@uoradea.ro, galsorin23@gmail.com}
\author{Constantin P. Niculescu}
\address{Department of Mathematics, University of Craiova\\
Craiova 200585, Romania and The Institute of Mathematics of the Romanian
Academy, Bucharest, Romania}
\email{constantin.p.niculescu@gmail.com}
\date{September 28, 2022}

\begin{abstract}
By extending the classical quantitative approximation results for positive and linear operators in $L^{p}([0, 1]), 1\le p \le +\infty$ of Berens and DeVore in 1978 and of Swetits and Wood in 1983 to the more general case of sublinear, monotone and strongly translatable operators, in this paper we obtain quantitative estimates in terms of the second order and third order moduli of smoothness, in Korovkin type theorems. Applications to concrete examples are included and an open question concerning interpolation theory for sublinear, monotone and strongly translatable operators is raised.
\end{abstract}
\maketitle

\section{Introduction}

Korovkin's theorem \cite{Ko1953}, \cite{Ko1960} provides a very simple test of
convergence to the identity for any sequence $(T_{n})_{n}$ of positive linear
operators that map $C\left(  [0,1]\right)  $ into itself: the occurrence of
this convergence for the functions $e_{0}(x)=1, e_{1}(x)=x$ and $e_{2}(x)=x^{2}$. In other words, the
fact that
\[
\lim_{n\rightarrow\infty}T_{n}(f)=f\text{\quad uniformly on }[0,1]
\]
for every $f\in C\left(  [0,1]\right)  $ reduces to the status of the three
aforementioned functions. Due to its simplicity and usefulness, this result
has attracted a great deal of attention leading to numerous generalizations.
Part of them are included in the authoritative monograph of Altomare and
Campiti \cite{AC1994} and the excellent survey of Altomare \cite{Alt2010}. See
\cite{Alt2021}, \cite{Alt2021b}, \cite{Alt2022}, for some very recent contributions.

In the case of positive and linear operators, quantitative estimates in terms of the uniform norm and modulus of continuity for the classical Korovkin's theorems were obtained by Shisha and Mond in \cite{SM}, and for the case of $L_{p}$ spaces with $1\le p <\infty$ were obtained by Berens and DeVore in \cite{BD} and by Swetits and Wood in \cite{SW1}.

Korovkin's theorem was extended to the framework of sublinear and monotone
operators acting on function spaces defined on appropriate subsets $K$ of
$\mathbb{R}^{N}$ in Gal and Niculescu \cite{Gal-Nic-Med}, \cite{Gal-Nic-Aeq},
\cite{Gal-Nic-RACSAM} and \cite{Gal-Nic-subm}.
Also, a quantitative estimate in terms of the uniform norm and uniform modulus of continuity in the operator version of Korovkin type theorems for the case of monotone and sublinear operators, was obtained in Gal and Niculescu  \cite{Gal-Nic-subm2}.

The goal of the present paper is to obtain quantitative estimates in Korovkin's theorems  in terms of uniform norm and the second order modulus of smoothness in $C([0, 1])$ and in terms of $L_p$ norm and second and third $L_p$ moduli of smoothness, with $1\le p \le \infty$. Our results are based on the extensions of some classical results for positive and linear operators of Berens and DeVore in \cite{BD} and of Swetits and Wood in \cite{SW1}, to the more general frame of sublinear, monotone and strongly translatable operators.

Section 2 contains some preliminaries on weakly nonlinear and monotone operators. The main results for $p=1$ and $p=\infty$ are proved in Section 3
by adapting the results for positive and linear operators in Berens and DeVore \cite{BD} to sublinear monotone and strongly translatable operators. The case $1 < p < \infty$ of Theorem 3 in Berens-DeVore \cite{BD} is based on an interpolation  technique for linear continuous  operators. Therefore an extension of Theorem 3 in Berens-DeVore \cite{BD} in our new frame, would require  an extension of this technique to sublinear monotone and strongly translatable operators. But because for the moment it seems that such a theory is not known, the
quantitative approximation results for $1 < p < \infty$ are obtained in Section 4 by adapting the results for positive and linear operators in Swetits and Wood \cite{SW1}, to sublinear monotone and strongly translatable operators.
Section 5 presents some applications of the main results obtained to some concrete examples. Section 6 raises as open question the extension of the interpolation technique to sublinear monotone and strongly translatable operators.

\section{Weakly nonlinear operators acting on ordered Banach spaces}

The goal of this section is to describe a class of nonlinear operators which provides a
convenient framework for the extension of Berens and DeVore's results in \cite{BD}.

Given a metric space $X,$ we attach to it the vector lattice $\mathcal{F}(X)$
of all real-valued functions defined on $X$, endowed with the metric $d$ and
the pointwise ordering.

Suppose that $X$ and $Y$ are two metric spaces and $E$ and $F$ are
respectively ordered vector subspaces (or subcones of the positive cones) of
$\mathcal{F}(X)$ and $\mathcal{F}(Y)$ and that $\mathcal{F}(X)$ contains the
unity. An operator $T:E\rightarrow F$ is said to be a \emph{weakly nonlinear}
if it satisfies the following two conditions:

\begin{enumerate}
\item[(SL)] (\emph{Sublinearity}) $T$ is subadditive and positively
homogeneous, that is,%
\[
T(f+g)\leq T(f)+T(g)\quad\text{and}\quad T(\alpha f)=\alpha T(f)
\]
for all $f,g$ in $E$ and $\alpha\geq0;$

\item[(TR)] (\emph{Translatability}) $T(f+\alpha\cdot1)=T(f)+\alpha T(1)$ for
all functions $f\in E$ and all numbers $\alpha\geq0.$
\end{enumerate}

In the case when $T$ is \emph{unital} (that is, $T(1)=1)$ the condition of
translatability takes the form%
\[
T(f+\alpha\cdot1)=T(f)+\alpha1,
\]
for all $f\in E$ and $\alpha\geq0.$

A stronger condition than translatability is

\begin{enumerate}
\item[(TR$^{\ast}$)] (\emph{Strong translatability}) $T(f+\alpha
\cdot1)=T(f)+\alpha T(1)$ for all functions $f\in E$ and all numbers
$\alpha\in\mathbb{R}.$
\end{enumerate}

The last condition occurs naturally in the context of Choquet's integral,
being a consequence of what is called there the property of \emph{comonotonic
additivity}, that is,

\begin{enumerate}
\item[(CA)] $T(f+g)=T(f)+T(g)$ whenever the functions $f,g\in E$ are
comonotone in the sense that%
\[
(f(s)-f(t))\cdot(g(s)-g(t))\geq0\text{\quad for all }s,t\in X.
\]
See \cite{Gal-Nic-Aeq} and \cite{Gal-Nic-JMAA}, as well as the references therein.
\end{enumerate}

In this paper we are interested in those called weakly nonlinear, which verify the following condition:

\begin{enumerate}
\item[(M)] (\emph{Monotonicity}) $f\leq g$ in $E$ implies $T(f)\leq T(g)$ for
all $f,g$ in $E.$
\end{enumerate}

\begin{remark}
\label{rem1} If $T$ is a weakly nonlinear and monotone operator, then
\[
T(\alpha\cdot1)=\alpha\cdot T(1)\text{\quad for all }\alpha\in\mathbb{R}.
\]
Indeed, for $\alpha\geq0$ the property follows from positive homogeneity.
Suppose now that $\alpha<0.$ Since $T(0)=0$ and $-\alpha>0$, by
translatability it follows that $0=T(0)=T(\alpha\cdot1+(-\alpha\cdot
1))=T(\alpha\cdot1)+(-\alpha)T(1)$, which implies $T(\alpha)=\alpha T(1)$.
\end{remark}

Ergodic theory, harmonic analysis, probability theory and Choquet's theory of integration offer numerous examples of
monotone sublinear and strongly translatable operators, see, e.g.,
Gal and Niculescu \cite{Gal-Nic-Aeq}, \cite{Gal-Nic-JMAA}, \cite{Gal-Nic-RACSAM}.

Suppose that $E$ and $F$ are two ordered Banach spaces and $T$ $:E\rightarrow
F$ is an operator (not necessarily linear or continuous).

If $T$ is positively homogeneous, then
\[
T(0)=0.
\]
As a consequence,
\[
-T(-f)\leq T(f)\text{\quad for all }f\in E
\]
and every positively homogeneous and monotone operator $T$ maps positive
elements into positive elements, that is,%
\begin{equation}
Tf\geq0\text{\quad for all }f\geq0. \label{pos-op}%
\end{equation}
Therefore, for linear operators the property (\ref{pos-op}) is equivalent to monotonicity.

Every sublinear operator is convex and a convex operator is sublinear if and
only if it is positively homogeneous.

The \emph{norm} of a continuous sublinear operator $T:E\rightarrow F$ can be
defined via the formulas%
\begin{align*}
\left\Vert T\right\Vert  &  =\inf\left\{  \lambda>0:\left\Vert T\left(
f\right)  \right\Vert \leq\lambda\left\Vert f\right\Vert \text{ for all }f\in
E\right\} \\
&  =\sup\left\{  \left\Vert T(f)\right\Vert :f\in E,\text{ }\left\Vert
f\right\Vert \leq1\right\}  .
\end{align*}
A sublinear operator may be discontinuous, but when it is continuous, it is
Lipschitz continuous. More precisely, if $T:E\rightarrow F$ is a continuous
sublinear operator, then
\[
\left\Vert T\left(  f\right)  -T(g)\right\Vert \leq2\left\Vert T\right\Vert
\left\Vert f-g\right\Vert \text{\quad for all }f\in E.
\]
All sublinear and monotone operators are Lipschitz continuous, as stated by the following result.

\begin{theorem}
\label{thmKrein}Every sublinear and monotone operator $T$ $:E\rightarrow F$
verifies the inequality
\[
\left\vert T(f)-T(g)\right\vert \leq T\left(  \left\vert f-g\right\vert
\right)  \text{\quad for all }f,g\in E
\]
and thus it is Lipschitz continuous with Lipschitz constant equals to
$\left\Vert T\right\Vert ,$ that is,
\[
\left\Vert T(f)-T(g)\right\Vert \leq\left\Vert T\right\Vert \left\Vert
f-g\right\Vert \text{\quad for all }f,g\in E.
\]
\end{theorem}

See \cite{Gal-Nic-subm} for details. Theorem \ref{thmKrein} is a
generalization of a classical result of M. G. Krein concerning the continuity
of positive linear functionals. See \cite{AA2001}.

\begin{remark}
The above weakly nonlinear operators were also called in \cite{Gal-Nic-Aeq} as Choquet operators. Notice that the classical integral representation of such operators were generalized by using the Choquet-Bochner integral of a real-valued function with respect to a vector capacity
in Gal and Niculescu \cite{Gal-Nic-JMAA}.
\end{remark}

\section{Main results for $p=\infty$ and $p=1$}

For $r\ge 0$ and $1\le p \le +\infty$, let us denote by $W_{p}^{r}([0, 1])$ the Sobolev space of all functions $f\in L_{p}([0, 1])$ such that the derivatives
$f^{(\alpha)}$ exist (in the Sobolev sense) and are in $L_{p}([0, 1])$, for all $\alpha\in \{0, 1, ... , r\}$, endowed with the norm
$$\|f\|_{r, p}=\max_{j=0, 1, ... , r}\|f^{(j)}\|_{p}.$$
Also, if $T:C([0, 1])\to C([0, 1])$ is a monotone sublinear operator, denote
$$\lambda_{p}=\max\{\|T(e_0)-e_0\|_{p}, \|T(e_1)-e_1\|_{p}, \|T(-e_1)+e_1\|_{p}, \|T(e_2)-e_2\|_{p}\}.$$

The first main result is the following.

\begin{theorem}\label{thm1} Let $1\le p \le +\infty$ and $T:C([0, 1])\to C([0, 1])$ be a sublinear monotone and strongly translatable operator. If $f\in W^{\infty}_{2}$, then
\begin{equation}\label{eq1}
\|f - T(f)\|_{p}\le C\cdot \|f\|_{2, \infty}\cdot \lambda_{p},
\end{equation}
with $C>0$ an absolute  constant.
\end{theorem}
\begin{proof} We will adapt the considerations for the positive linear operator in the proof of Theorem 1 in \cite{BD}, to the properties of the monotone sublinear and strongly translatable operators.

Firstly, let us consider that $1\le p <+\infty$.

Choose $\varepsilon=\frac{1}{k}$, $k\in \mathbb{N}$ and write $(0, 1)$ as an union of $k$ subntervals $I_i$, pairwise disjoint, of lengths $\le \frac{1}{k}$.
For each $i\in \{1, ..., k\}$ let $\xi_{i}$ be the center of $I_i$ and define
$$l_{i}(x)=f(\xi_{i})+f^{\prime}(\xi_i)(x-\xi_{i}).$$
Reasoning exactly as for the relation $(2.3)$ in the paper of Berens-DeVore, we arrive at the estimate
\begin{equation}\label{eq2}
|f(x)-l_{i}(x)|\le \frac{1}{2}\|f\|_{2, \infty}|x-\xi_{i}|^{2}.
\end{equation}
Then, since for each $i\in \{1, ..., k\}$ and almost all $x\in I_{i}$ we have
$$T(f)(x)-f(x)=T(f)(x)-T(l_i)(x)+T(l_i)(x)-l_i(x)+l_i(x)-f(x)$$
and taking into account the property of $T$ in Theorem 1 too, it follows
$$|T(f)(x)-f(x)|\le T(|f-l_i|)(x)+|T(l_i)(x)-l_i(x)|+|l_i(x)-f(x)|$$
and consequently for almost everywhere  $x\in [0, 1]$ we obtain
$$|T(f)(x)-f(x)|$$
$$\le \sum_{i=1}^{k}T(|f-l_i|)(x)\chi_{I_i}(x)+\sum_{i=1}^{k}|T(l_i)(x)-l_i(x)|\chi_{I_i}(x)+\sum_{i=1}^{k}|l_i(x)-f(x)|\chi_{I_i}(x)$$
$$:=S_1(x)+S_2(x)+S_3(x).$$
By inequality (\ref{eq2}), for almost everywhere $x\in [0, 1]$, we obtain
$$S_1(x)\le \frac{1}{2}|\|f\|_{2, \infty}T(|t-\xi_i|^{2})(x).$$
But by the sublinearity of $T$, we get
$$T(|t-\xi_i|^{2})(x)=T[e_2(t)+2 (-e_1)(t) \xi_i + \xi_i^{2})e_{0}(t)](x)$$
$$\le T(e_2)(x)+2 \xi_i T(-e_1)(x)+\xi_i^{2}T(e_0)(x)$$
$$=T(e_2)(x)-e_2(x)+2 \xi_i [T(-e_1)(x)+e_1(x)]+\xi_{i}^{2}[T(e_0)(x)-e_0(x)]+(x-\xi_i)^{2}$$
$$\le |T(e_2)(x)-e_2(x)|+2  |T(-e_1)(x)+e_1(x)|+|T(e_0)(x)-e_0(x)|+(x-\xi_i)^{2},$$
which immediately implies (as in the proof for the estimate (2.6) in Berens and DeVore \cite{BD})
$$\|S_1\|_{p}\le \frac{1}{2}\le \|f\|_{2, \infty}(4 \lambda_p+ \varepsilon^{2}).$$
For the estimate of $S_2(x)$, for each $i\in \{1, ..., k\}$ and almost all $x\in I_i$, we get
$$|T(l_i)(x)-l_i(x)|=|T[f(\xi_i)-f^{\prime}(\xi_i) \xi_i + f^{\prime}(\xi_i)e_1(t)](x)-f(\xi_i)+f^{\prime} (\xi_{i})\xi_i -x f^{\prime}(\xi_i)|$$
$$=|(f(\xi_i)-f^{\prime}(\xi_i) \xi_i)[T(e_0)(x)-1]+T(f^{\prime}(\xi_i)e_1)(x)-x f^{\prime}(\xi_i)|$$
$$\le (\|f\|_{\infty}+\|f^{\prime}\|_{\infty})|T(e_0)(x)-1|+|T(f^{\prime}(\xi_i)e_1)(x)-x f^{\prime}(\xi_i)|.$$
Now, if $f^{\prime}(\xi_i)\ge 0$, then
$$|T(f^{\prime}(\xi_i)e_1)(x)-x f^{\prime}(\xi_i)|=f^{\prime}(\xi_i) |T(e_1)(x)-x|\le \|f^{\prime}\|_{\infty}|T(e_1)(x)-x|$$
and if f $f^{\prime}(\xi_i)< 0$, then
$$|T(f^{\prime}(\xi_i)e_1)(x)-x f^{\prime}(\xi_i)|$$
$$=|T(-f^{\prime}(\xi_i)(-e_1))(x)+x (-f^{\prime}(\xi_i))|=-f^{\prime}(\xi_i)|T(-e_1)(x)+x|$$
$$\le \|f^{\prime}\|_{\infty}|T(-e_1)(x)+x|.$$
From these estimates it immediately follows that
$$|T(l_i)(x)-l_i(x)|$$
$$\le (\|f\|_{\infty}+\|f^{\prime}\|_{\infty})[|T(e_0)(x)-1|+|T(e_1)(x)-x|+|T(-e_1)(x)+x|]$$
$$\le 2 \|f\|_{1, \infty}[|T(e_0)(x)-1|+|T(e_1)(x)-x|+|T(-e_1)(x)+x|+|T(e_2)(x)-x^2|],$$
which leads to
$$\|S_2\|_{p}\le 2 \|f\|_{1, p} \lambda_{p}.$$
Finally, from (\ref{eq2}) it easily follows that
$$\|S_3\|_{p}\le \frac{1}{2}\|f\|_{2, \infty} \varepsilon^{2}.$$
Collecting now all the estimates for $S_1, S_2, S_3$ and taking into account that $\varepsilon>0$ is arbitrary small, we arrive at
the estimate in the statement.

The case $p=+\infty$ easily follows by using the above lines of proof in the case when $1\le p <+\infty$.
\end{proof}

If we define the $r$th order modulus of smoothness in $L_{p}([0, 1])$, $1\le p\le \infty$, by
$$\omega_{r, p}(f; \delta)=\sup_{|h|\le \delta}\|\Delta_{h}^{r}\|_{p},$$
where $\Delta_{h}^{r}(x)=\sum_{j=0}^{r}(-1)^{j}{r \choose j}f(x+(r-j)h)$, we can state the following result which is an analogue of Theorem 2 in Berens and DeVore \cite{BD}.

\begin{theorem}\label{thm2} (i) If $T:C([0, 1])\to C([0, 1])$ is a sublinear monotone and strongly translatable operator, then for all
$f\in C([0, 1])$ we have
$$\|f - T(f)\|_{\infty}\le C\{\|f\|_{\infty} \lambda_{\infty}+ \omega_{2, \infty}(f; \lambda_{\infty}^{1/2})\},$$
where $C$ depends only on the norm of $T$.

(ii) Let $p=1$ and $T:L_{1}([0, 1])\to L_{1}([0, 1])$ be a monotone sublinear and strongly translatable operator. Then for all $f\in L_{1}([0, 1])$ we have
$$\|f - T(f)\|_{1}\le C\{\|f\|_{1} \lambda_{1}+\omega_{3, 1}(\lambda_{1}^{1/3})\},$$
where $C$ depends only on the norm of $T$.
\end{theorem}
\begin{proof}
The proof is based on the above Theorem \ref{thm1} and clearly that it is identical with the proof of Theorem 2 in Berens and DeVore \cite{BD}, based also on Lemma 1 in the same paper.

For the reader's convenience, we sketch below the proof. For $1\le p\le +\infty$ and $r\in \mathbb{N}$, let us consider the $K$ functional
$$K_{r, p}(f; t)=\inf\{\|f-g\|_{p} + t \|g\|_{p}; g\in W_{p}^{r}([0, 1])\}, t >0.$$
It is well-known the fact (see, e.g., relation (3.1) in \cite{BD}) that we have
\begin{equation}\label{eq3}
K_{r, p}(t; t)\le C(t^{r}\|f\|_{p}+\omega_{r, p}(f; t)),
\end{equation}
with $C>0$ depending only on $r$.

(i) For $p=\infty$, from the estimate in Theorem \ref{thm1} in this paper combined with (\ref{eq3}), we easily get
$$\|f-T(f)\|_{\infty}\le C_{T}\cdot K_{2, \infty}(f; \lambda_{\infty})\le C_{T}\{\lambda_{\infty}\|f\|_{\infty} +\omega_{2, \infty}(f; \lambda_{\infty}^{1/2})\}.$$

(ii) For $p=1$, Lemma 1 in \cite{BD} states that for all $f\in W_{1}^{1}$ we have $\|f\|_{\infty}\le \|f\|_{1, 1}$. From here, from the estimate in Theorem\ref{thm1} and from (\ref{eq3}), we obtain
$$\|g - T(g)\|_{1}\le C\|g\|_{2, \infty} \lambda_{1} \le \|g\|_{3, 1} \lambda_{1}, \mbox{ for each } g\in W_{1}^{3}$$
and consequently
$$\|f - T(f)\|_{1}\le C_{T} \cdot K_{3, 1}(f; \lambda_{1})\le C_{T}(\lambda_{1}\|f\|_{1}+\omega_{3, \infty}(f; \lambda_{1}^{1/3})).$$
\end{proof}

\section{Main results for $1 < p < \infty$}

Due to the reason mentioned at the end of Introduction, in this section we adapt the proofs in Swetits-Wood \cite{SW1}, to the case of  sublinear monotone and strongly translatable operators.

For the proof of the main result we need the following auxiliary result.

\begin{lemma}\label{lem1}
Let $(L_n)_{n}$ be a uniformly bounded sequence of sublinear, strongly translatable and monotone operators from $L_{p}([0, 1])$ into $L_{p}([0, ])$, where $ 1 < p < \infty$. Let $L_{p}^{(2)}([0, 1])$ be the space of those functions $f\in L_{p}([0, 1])$ with $f^{\prime}$ absolutely continuous and $f^{\prime \prime}\in L_{p}([0, 1])$.
Let us denote $\mu_{n}=\|T_{n}((e_1-x)^{2})(x)\|_{\infty}$ and
$$t_{n, p}=(\max\{\|L(e_0)-e_0\|_{p}, \|T_{n}(|e_1-x|)(x)\|_{p}, \mu_{n}\})^{1/2}.$$

If $\lim_{n\to \infty}t_{n, p}=0$, then for any $f\in L_{p}^{(2)}([0, 1])$, we have
$$\|T_{n}(f) - f\|_{p}\le M_{p}^{\prime}(\|f\|_{p}+\|f^{\prime \prime}\|_{p})t^{2}_{n, p},$$
where $M^{\prime}_{p} > 0$ is independent of $f$ and $n$.
\end{lemma}

\begin{proof} Let $f\in L_{p}^{(2)}([0, 1])$ and assume that $f$ has  been extended outside of $[0, 1]$ so that $f^{\prime \prime}(x)=0$ if $x\notin [0, ]$.

By
$$T_{n}(f)(x)-f(x)=T_{n}(f)(x)-f(x)T_{n}(e_0)(x)+f(x)T_{n}(e_0)(x)-f(x)$$
$$=T_{n}(f)(x)-f(x)T_{n}(e_0)(x)+f(x)(T_{n}(e_0)(x)-1),$$
we obtain
\begin{equation}\label{eq4}
\|T_{n}(f)(x)-f(x)\|_{p}
\end{equation}
$$\le \|T_{n}(f)(x)-f(x)T_{n}(e_0)(x)\|_{p}+\|f\|_{\infty}\cdot \|T_{n}(e_0)(x)- e_{0}(x)\|_{p}.$$

For $t, x\in [0, 1]$, we have
$$f(t)-f(x)=f(t)-f(x)e_{0}(t)=f^{\prime}(x)(t-x)+\int_{x}^{t}(t-u)f^{\prime \prime}(u)du$$
which by applying to the both members $T_{n}$, implies
$$T_{n}(f)(x)-f(x)T_{n}(e_0)(x)=T_{n}\left [f^{\prime}(x)(e_1-x)+\int_{x}^{t}(t-u)f^{\prime \prime}(u)du\right ].$$
Taking the absolute value, by Theorem \ref{thmKrein} and by the sublinearity of $T_{n}$, we get
$$|T_{n}(f)(x)-f(x)T_{n}(e_0)(x)|=\left |T_{n}\left [f^{\prime}(x)(e_1-x)+\int_{x}^{t}(t-u)f^{\prime \prime}(u)du\right ]\right |$$
$$\le T_{n}\left [|f^{\prime}(x)|\cdot |e_1-x|+|\int_{x}^{t}(t-u)f^{\prime \prime}(u)du|\right ]$$
$$\le \|f^{\prime}\|_{\infty}T_{n}(|e_1-x|)(x)+T_{n}\left [\left |\int_{x}^{t}(t-u)f^{\prime \prime}(u)du\right |\right ].$$
This implies
\begin{equation}\label{eq5}
\|T_{n}(f)(x)-f(x)T_{n}(e_0)(x)\|_{p}
\end{equation}
$$\le \|f^{\prime}\|_{\infty}\|T_{n}(|e_1-x|)(x)\|_p+\left \|T_{n}\left [\left |\int_{x}^{t}(t-u)f^{\prime \prime}(u)du\right |\right ](x)\right \|_{p}.$$
Reasoning exactly as in Swetits-Wood \cite{SW1}, proof of Lemma 2, page 87 (that is by using the Hardy-Littlewood majorant), we get
$$\left \|T_{n}\left [\left |\int_{x}^{t}(t-u)f^{\prime \prime}(u)du\right |\right ](x)\right \|_{p}\le K_p \cdot \|T_{n}((e_1-x)^{2})(x)\|_{\infty}\cdot \|f^{\prime \prime}\|_{p}.$$
Therefore, (\ref{eq5}) becomes
\begin{equation}\label{eq6}
\|T_{n}(f)(x)-f(x)T_{n}(e_0)(x)\|_{p}
\end{equation}
$$\le \|f^{\prime}\|_{\infty}\cdot \|\left [T_{n}(|e_1-x|)(x)\right ]\|_{p}+
K_p \cdot \|T_{n}((e_1-x)^{2})(x)\|_{\infty}\cdot \|f^{\prime \prime}\|_{p}.$$
From (\ref{eq4}), it immediately follows
\begin{equation}\label{eq7}
\|T_{n}(f)(x)-f(x)\|_{p}\le \|f\|_{\infty}\cdot \|T_{n}(e_0)(x)- e_{0}(x)\|_{p}
\end{equation}
$$+\|f^{\prime}\|_{\infty}\cdot \|T_{n}(|e_1-x|)(x)\|_{p}+
K_p \cdot \|T_{n}((e_1-x)^{2})(x)\|_{\infty}\cdot \|f^{\prime \prime}\|_{p}.$$
Using now Theorem 3.1 in \cite{Gold}, from (\ref{eq7}) we immediately arrive to
\begin{equation}\label{eq8}
\|T_{n}(f)(x)-f(x)\|_{p}\le C^{\prime}_{p}(\|f\|_{p}+\|f^{\prime \prime}\|_{p})t_{n, p}^{2}.
\end{equation}
\end{proof}

The main result of this section is the following.

\begin{theorem}\label{thmfinal} Let $(T_n)_{n}$ be a uniformly bounded sequence of strongly translatable sublinear and monotone operators from $L_{p}([0, 1])$ into $L_{p}([0, 1])$, where $ 1 < p < \infty$, and denote $\mu_{n}=\|T_{n}((e_1-x)^{2})(x)\|_{\infty}$,
$$t_{n, p}=(\max\{\|T_{n}(e_0)-e_0\|_{p}, \|T_{n}(|e_1-x|)(x)\|_{p}, \mu_{n}\})^{1/2},$$

Supposing that $\lim_{n\to \infty}t_{n, p}=0$, for any $f\in L_{p}([0, ])$, we have
$$\|T_{n}(f) - f\|_{p}\le M_{p}[t^{2}_{n, p}\|f\|_{p}  +\omega_{2, p}(f; t_{n, p})],$$
where $M_p >0$ is independent of $f$ and $n$.
\end{theorem}
\begin{proof}
Let $f\in L_{p}([0, 1])$ and $g\in L_{p}^{(2)}([0, 1])$. By Theorem \ref{thmKrein}, for each  $T_{n}$ we have
$$\|T_{n}(f)-T_{n}(g)\|_{p}\le \|T_{n}\|\cdot \|f - g\|_{p}, \mbox{ for all } f, g\in L_{p}([0, 1])\le R_{p}\cdot \|f-g\|_{p},$$
where $\|T_{n}\|\le R_{p}$ (with $R_{p}$ independent of $n$) for all $n\in \mathbb{N}$, from the uniform boundedness of the sequence $(T_{n})_{n}$.

Then, by
$$L_{n}(f)-f=L_{n}(f)-L_{n}(g)+L_{n}(g)-g + g - f,$$
passing to absolute value and to $\|\cdot\|_{p}$, we immediately get (by Lemma \ref{lem1} too)
$$\|L_{n}(f)-f\|_{p}\le \|L_{n}(f)-L_{n}(g)\|_{p}+\|L_{n}(g)-g\|_{p} + \|g - f\|_{p}$$
$$\le (1+R_{p})\|f-g\|_{p}+M^{\prime}_{p}\cdot t^{2}_{n, p}(\|g\|_{p}+\|g^{\prime \prime}\|_{p}).$$
Taking here the infimum after $g\in L_{p}^{(2)}([0, 1])$ and taking into account relations $(2.1)$ and $(2.3)$ on the page 88 in Swetits-Wood \cite{SW1}, we arriveat the desired conclusion.
\end{proof}

\begin{remark}
Since the calculation of $L_{n}(|e_1 - x|)(x)$ is difficult, we may estimate it by a simpler quantity in calculation,  by using the H\"older's inequality in Theorem 3 of Gal and Niculescu \cite{Gal-Nic-Aeq}. It follows
$$T_{n}(|e_1-x|)(x)\le (T_{n}((e_1-x)^{2})(x) T_{n}(e_0)(x))^{1/2}\le C^{1/2}(T_{n}((e_1-x)^{2})(x))^{1/2},$$
where $C>0$ is a constant independent of $n$ which comes from the uniform boundedness of the sequence $(T_{n})_{n}$.

This implies
$$\|T_{n}(|e_1-x|)(x)\|_{p}\le C^{1/2}\|\left [T_{n}((e_1-x)^{2})(x)\right ]^{1/2}\|_{p}.$$
\end{remark}
By this remark, we immediately arrive at the following result.
\begin{corollary}\label{coro}
Let $(T_n)_{n}$ be a uniformly bounded sequence of strongly translatable sublinear and monotone operators from $L_{p}([0, 1])$ into $L_{p}([0, 1])$, where $ 1 < p < \infty$, and denote $\mu_{n}=\|T_{n}((e_1-x)^{2})(x)\|_{\infty}$,
$$s_{n, p}=(\max\{\|T_{n}(e_0)-e_0\|_{p}, \|[T_{n}((e_1-x)^{2})(x)]^{1/2}\|_{p}, \mu_{n}\})^{1/2},$$

Supposing that $\lim_{n\to \infty}s_{n, p}=0$, for any $f\in L_{p}([0, 1])$, we have
$$\|T_{n}(f) - f\|_{p}\le M_{p}[s^{2}_{n, p}\|f\|_{p}  +\omega_{2, p}(f; s_{n, p})],$$
where $M_p >0$ is independent of $f$ and $n$.
\end{corollary}
\begin{remark} \label{rem5}
If $(T_{n})_{n}$ is a sequence of strongly translatable sublinear and monotone operators from $L_{p}([0, 1])$ into $L_{p}([0, 1])$ which satisfies the conditions
$$\lim_{n\to \infty}\|T_{n}(e_0)-e_0\|_{\infty}=0, \lim_{n\to \infty}\|T_{n}(-e_1)+e_1\|_{\infty}=0, \lim_{n\to \infty}\|T_{n}(e_2)-e_2\|_{\infty}=0,$$
then $\lim_{n\to \infty}s_{n, p}=0$.

Indeed, this is immediate from the relations
$$0\le T_{n}((e_1-x)^{2})(x)=T_{n}[e_2 -2x e_1 + x^2e_0](x)$$
$$\le T_n(e_2)(x)+2x T_{n}(-e_1)(x)+x^{2}T_{n}(e_0)(x)\to 0 \mbox{ as } n\to \infty$$
and
$$\|[T_{n}((e_1-x)^{2})(x)]^{1/2}\|_{p}\le \|[T_{n}((e_1-x)^{2})(x)]^{1/2}\|_{\infty}.$$
\end{remark}

\section{Applications}

In this section we apply the previous results to some concrete cases.

\begin{example}\label{ex1}
Let us consider the sequence of Bernstein operators $B_{n}%
:C([0,1])\rightarrow C\left(  [0,1]\right)  ,$ defined by the formulas%
\[
B_{n}(f)(x)=\sum_{k=0}^{n}p_{n,k}(x)f(k/n).
\]
The operators $T_{n}:C([0,1])\rightarrow C\left (  [0,1]\right)  $
given by
\[
T_{n}(f)=\max\left\{  B_{n}(f),B_{n+1}(f)\right\}
\]
are sublinear, monotone and strongly translatable. Known computations imply that
$$B_{n}(e_0)(x)=1, B_{n}(-e_1)(x) = -e_1(x), B_{n}(e_1)=e_1(x), B_n(e_2)(x)=e_2(x) +\frac{x(1-x)}{n}.$$
These imply $T_{n}(e_0)=e_0$, $T_{n}(-e_1)=-e_1$, $T_{n}(e_1)=e_1$ and
$$T_{n}(e_2)(x)=\max\{B_{n}(e_2)(x), B_{n+1}(e_2)(x)\}=e_2(x)+\frac{x(1-x)}{n}.$$

Since for $p=\infty$ we immediately get
$$\lambda_{n, p}=\|x(1-x)/n\|_{p}\le \frac{1}{4 n},$$
by Theorem \ref{thm2}, (i), it follows the estimate
$$\|f - T_{n}(f)\|_{\infty}\le C\left [\|f\|_{\infty} \cdot \frac{1}{4 n}+\omega_{2, \infty}\left (f; \frac{1}{2 \sqrt{n}}\right )\right ],$$
which is completely different and essentially better than the estimate based on the Shisha and Mond's idea \cite{SM} in the paper Gal and Niculescu \cite{Gal-Nic-subm2}, namely
$$\|f-T_{n}(f)\|_{\infty}\le 2 \omega_{1, \infty}(f; 1/(2\sqrt{n})).$$
For $p=1$, by Theorem \ref{thm2}, (ii), we get
$$\|f - T_{n}(f)\|_{1}\le C\left [\|f\|_{1} \cdot \frac{1}{4 n}+\omega_{3, 1}\left (f; \frac{1}{(4 n)^{1/3}}\right )\right ].$$
For $1 < p < \infty$, since $T_{n}((e_1-x)^{2})(x)=\frac{x(1-x)}{n}$, $\|T_n((e_1-x)^{2})(x)\|_{\infty}\le \frac{1}{4 n}$ and
$\|[T_{n}((e_1-x)^{2})(x)]^{1/2}\|_{\infty}\le \frac{1}{2 \sqrt{n}}$,
by Corollary \ref{coro} and by Remark \ref{rem5} we obtain
$$\|T_{n}(f)-f\|_{p}\le M_{p}\left [\|f\|_{p}\cdot \frac{1}{2\sqrt{n}} + \omega_{2, p}\left (f; \frac{1}{\sqrt{2}\cdot n^{1/4}}\right )\right ].$$
\end{example}

\begin{example}\label{ex2}
Now, let us define the sequence of nonlinear operators $T_{n}:C\left(
[0,1]\right)  \rightarrow C\left(  [0,1]\right)  $ defined by the formulas
\[
T_{n}(f)(x)=\sum_{k=0}^{n}p_{n,k}(x)\sup_{[k/(n+1)\leq t\leq
,(k+1)/(n+1)]}f(t),
\]
which are monotone sublinear and strongly translatable. We have
$T_{n}(e_0)(x)=1$, $T_{n}(e_1)(x)=\frac{n}{n+1}e_{1}(x)$,
$T_{n}(-e_1)(x)=-\frac{n}{n+1} e_{1}(x)$,
$$\|T_{n}(e_1)-e_1\|_{p}\le \frac{1}{n+1}, \|T_{n}(-e_1)+e_1\|_{p}\le \frac{1}{n+1},$$
$$T_{n}(e_2)(x)=\sum_{k=0}^{n}p_{n, k}(x)\cdot \frac{(k+1)^{2}}{(n+1)^{2}}$$
$$=\left(\frac{n}{n+1}\right )^{2}\left (x^{2}+\frac{x(1-x)}{n}\right )+\frac{2 n}{(n+1)^{2}}x+\frac{1}{(n+1)^{2}},$$
$$\|T_{n}(e_2)-e_2\|_{p}=\left \|e_2 \frac{-3n -1}{(n+1)^{2}}+\frac{3n x+1}{(n+1)^{2}}\right \|_{p}\le \frac{6 n+2}{(n+1)^{2}}\le \frac{6}{n+1},$$
and then
$$\lambda_{n, p}\le \frac{6}{n+1}.$$
By  Theorem \ref{thm2}, (i), we get
$$\|f - T_{n}(f)\|_{\infty}\le C\left [\|f\|_{\infty} \cdot \frac{6}{n+1}+\omega_{2, \infty}\left (f; \frac{\sqrt{6}}{\sqrt{n+1}}\right )\right ],$$
and by Theorem \ref{thm2}, (ii), it follows
$$\|f - T_{n}(f)\|_{1}\le C\left [\|f\|_{1} \cdot \frac{6}{n+1}+\omega_{3, 1}\left (f; \frac{6}{(n+1)^{1/3}}\right )\right ].$$
For $1 < p < \infty$, since by Remark \ref{rem5} we have
$$0\le T_{n}((e_1-x)^{2})(x)\le T_n(e_2)(x)+2x T_{n}(-e_1)(x)+x^{2}T_{n}(e_0)(x)$$
$$=\left(\frac{n}{n+1}\right )^{2}\left (x^{2}+\frac{x(1-x)}{n}\right )+\frac{2 n}{(n+1)^{2}}x+\frac{1}{(n+1)^{2}}+2x^{2}\left (-\frac{n}{n+1}\right )+x^{2}$$
which by simple calculation finally leads to
$$0\le T_{n}((e_1-x)^{2})(x)\le \frac{1}{n}\cdot \frac{9}{4},$$
by Corollary \ref{coro} we obtain
$$\|T_{n}(f)-f\|_{p}\le M_{p}\left [\|f\|_{p}\cdot \frac{3}{2\sqrt{n}} + \omega_{2, p}\left (f; \frac{\sqrt{3}}{\sqrt{2}\cdot n^{1/4}}\right )\right ].$$
\end{example}
\begin{example}
Another example can be the Bernstein-Kantorovich-Choquet operators, given by the formula
\[
K_{n, \mu}(f)(x)=\sum_{k=0}^{n}p_{n,k}(x)\cdot\frac{(C)\int_{k/(m+1)}%
^{(k+1)/(n+1)}f(t)\mathrm{d}\mu(t)}{\mu([k/(n+1),(k+1)/(n+1)])},
\]
where $(C)\int d\mu$ means the Choquet integral with respect to $\mu=\sqrt{m}$, with $m$ the Lebesgue measure.
We omit here the calculations. For details concerning the properties of the Choquet integral and of  Bernstein-Kantorovich-Choquet operators, see, e.g., Gal-Niculescu \cite{Gal-Nic-Med}, \cite{Gal-Nic-Aeq}.
\end{example}

\section{Final remarks}

\begin{remark}\label{rem6} The estimate in Theorem  \ref{thmfinal} is worse than that for linear and positive operators in Theorem 2, (i) in \cite{SW1}, where the quantity $\|T_{n}(|e_1-x|)(x)\|_{p}$ is replaced by the smaller one $\|T_{n}(e_1-x)(x)\|_{p}$, since obviously $\|T_{n}(e_1-x)(x)\|_{p}\le \|T_{n}(|e_1-x|)(x)\|_{p}$. This seems to be the price paid due to the more general hypothesis on the operators $L_{n}$ in the above Theorem \ref{thmfinal}.
\end{remark}
However, in the case of Example \ref{ex1}, since it is easy to show that
$$\max\{B_n(f), B_{n+1}(f)\} - f = \max\{B_n(f) - f, B_{n+1}(f) - f\},$$
and
$$|\max\{B_n(f) - f, B_{n+1}(f) - f\}|\le \max\{|B_n(f) - f|, |B_{n+1}(f) - f|\} ,$$
we immediately get
$$\|\max\{B_n(f), B_{n+1}(f)\} - f\|_{p} \le \max\{\|B_n(f) - f\|_{p}, \|B_{n+1}(f) - f\|_{p}\}$$
$$\le \|B_n(f) - f\|_{p}+\|B_{n+1}(f) - f\|_{p}.$$
Therefore, applying here the estimate in Swetits-Wood for $1<p<\infty$ and for the linear and positive operators $B_{n}$, clearly that we get essentially better estimate than that in Theorem \ref{thmfinal} and Corollary \ref{coro}, calculated in the above Example \ref{ex1}.

In the case of Example \ref{ex2}, we can write
$$T_{n}(f)(x)=\sum_{k=0}^{n}p_{n, k}(x)f(\xi_{n, k}),$$
with $k/(n+1)\le \xi_{n, k}\le (k+1)/(n+1)$, which immediately leads to
$$\|T_{n}(f)(x)-B_{n}(f)(x)|\le \sum_{k=0}^{n}p_{n, k}(x)|f(\xi_{n, k})-f(k/n)|\le \omega_{1}(f; 1/(n+))$$
and
$$|T_{n}(f)(x)-f(x)|\le |T_{n}(f)(x) - B_{n}(f)(x)|+|B_{n}(f)(x)-f(x)|.$$
By the result in Swetits-Wood applied to the positive linear operators $B_{n}(f)(x)$, for $1<p<\infty$, it immediately follows
$$\|T_{n}(f)-f\|_{p}\le M_p\left (\frac{1}{n}\|f\|_{p}+\omega_{2, p}(f;1/\sqrt{n})_{p}\right )+\omega(f; 1/(n+1)),$$
which clearly it is essentially better than the estimate in the previous section.

These  considerations suggest that possibly the shortcoming in Remark \ref{rem6} is due to the method of proof in Swetits-Wood \cite{SW1} and raise the following.

{\bf Open Question.} Extend the classical interpolation technique from linear continuous operators  to sublinear, monotone and strongly translatable operators.

A positive answer would allow to essentially improve the estimate  in Theorem \ref{thmfinal} by following now the method of proof in Theorem 3 of Berens-DeVore \cite{BD}.

\end{document}